\documentclass[letterpaper,11pt]{amsart}

\usepackage[all]{xy}                        %

\CompileMatrices                            

\UseTips                                    

\input xypic
\usepackage[bookmarks=true]{hyperref}       

\usepackage{amssymb,latexsym,amsmath,amscd}
\usepackage{xspace}

\usepackage{graphicx}

\reversemarginpar

\theoremstyle{plain}
\newtheorem{theorem}{Theorem}[section]
\newtheorem*{theorem*}{Theorem}
\newtheorem{proposition}[theorem]{Proposition}
\newtheorem{corollary}[theorem]{Corollary}
\newtheorem{lemma}[theorem]{Lemma}

\theoremstyle{definition}

\newtheorem{remark}[theorem]{Remark}

\newcommand{\enm}[1]{\ensuremath{#1}}          %
\newcommand{\op}[1]{\operatorname{#1}}
\newcommand{\cal}[1]{\mathcal{#1}}

\newcommand{\PP}{\enm{\mathbb{P}}}

\newcommand{\Ee}{\enm{\cal{E}}}
\newcommand{\Ff}{\enm{\cal{F}}}

\newcommand{\Hh}{\enm{\cal{H}}}

\newcommand{\Oo}{\enm{\cal{O}}}

\newcommand{\Rr}{\enm{\cal{R}}}

\newcommand{\Ww}{\enm{\cal{W}}}

\renewcommand{\phi}{\varphi}
\renewcommand{\theta}{\vartheta}
\renewcommand{\epsilon}{\varepsilon}

\newcommand{\Pic}{\op{Pic}}

\newcommand{\Hom}{\op{Hom}}

\newcommand{\Sec}{\op{Sec}}

\newcommand{\Image}{\op{Im}}

\renewcommand{\to}[1][]{\xrightarrow{\ #1\ }}

\newcommand{\old}[1]{}

\begin{document}

\title[Dominance of a Rational Map to the Coble Quartic]{Dominance of a Rational Map \\to the Coble Quartic}
\author{Sukmoon Huh}
\address{CIRM \\
Fondazione Bruno Kessler \\
Via Sommarive, 14-Povo\\
38100-Trento}
\email{sukmoon.huh@math.unizh.ch}
\thanks{This article is part of the revised version of the author's thesis at the
University of Michigan. The author would like to express his deepest
gratitude to his advisor Professor Igor Dolgachev. The author is also grateful to the referee for many suggestions.}


\maketitle
\begin{abstract} We show the dominance of the restriction map from a
moduli space of stable sheaves on the projective plane to the Coble
sixfold quartic. With the dominance and the interpretation of a stable sheaf
on the plane in terms of hyperplane arrangements, we expect these tools to reveal the geometry of the Coble quartic.
\end{abstract}

\section{Introduction}
Let $C$ be a smooth non-hyperelliptic curve of genus 3 over complex
numbers, then $C$ is embedded into $\PP_2 \simeq \PP H^0(K_C)^*$ by
canonical embedding as a plane quartic curve. The moduli space
$SU_C(2, K_C)$ of semistable vector bundles of rank 2 with canonical
determinant over $C$ is known to be a hypersurface in $\PP_7$,
called the `Coble quartic' \cite{Co}\cite{NR}. Let $\Ww^r$ be the
closure of the following set
\begin{equation}
\{ E \in SU_C(2,K_C) ~|~ h^0(C,E)\geq r+1\}.
\end{equation}
Then we have the following inclusions \cite{OPP} on the
Brill-Noether loci,
\begin{equation}
SU_C(2, K_C) \supset \Ww \supset \Ww^1 \supset \Ww^2 \supset \Ww^3
=\emptyset ,
\end{equation}
where $\Ww=\Ww^0$. Many properties on the geometry of these
Brill-Noether loci have been discovered in \cite{OPP}.

 Let $\overline{M}(c_1, c_2)$ be the moduli space of stable sheaves
of rank 2 with the Chern classes $(c_1, c_2)$ on the projective
plane. The dimension of this space is known to be $4c_2-3$ if $c_1=0$ \cite{Barth}, and $4c_2-4$ if $c_1=-1$ \cite{Hulek}. Then there exists a rational map \cite{Huh}
\begin{equation}
\Phi_k : \overline{M}(1, k) \dashrightarrow SU_C(2, K_C)~~~~,~~~~
1\leq k \leq 4
\end{equation}
defined by sending $E$ to $E|_C$. It is shown in \cite{Huh} that
$\Phi_k$ is a dominant map to $\Ww^2, \Ww^1$ and $\Ww$, for
$k=1,2,3$, respectively. In this article, we give a proof of the
dominance of the rational map $\Phi_4$. This is equivalent to the
dominance of the rational map from $\overline{M}(3,6)$ to $SU_C(2,
3K_C)$ by twisting. For a general bundle $E\in SU_C(2, 3K_C)$, we
embed $C$ with $\PP_2$ into a Grassmannian $Gr(5,2)$ and take the
pull-back of the universal quotient bundle of $Gr(5,2)$ to $\PP_2$.
This bundle is shown to be stable and have the Chern classes (3,6).

As a quick consequence, we can obtain the old result that $SU_C(2,
K_C)$ is unirational since $\overline{M}(1,4)$ is rational. The
unirationality implies the rationally connectedness. We see how we
can obtain a rational curve through two general points of the Coble
quartic in terms of hyperplane arrangements.

The restriction of vector bundles on $\PP_2$ to plane curves, was also studied in \cite{Hein}, where the author investigated the restriction of the tangent bundle of $\PP_2$ to plane curves and gave the conditions for a vector bundle $E$ on a plane curve to be a pull-back of the tangent bundle of $\PP_2$, twisted by $\Oo_{\PP_1}(-1)$.
 
 For the background on vector bundles, we suggest \cite{Mukai} as a good reference.

\section{Embedding Plane Quartics in Grassmannians}
Let $E$ be a semistable vector bundle of rank 2 with the determinant
$3K_C$ over $C$, i.e. $E\in SU_C(2,3K_C)$. By the following lemma,
we can obtain a morphism
$$\phi : C\rightarrow Gr(H^0(E), 2)$$
sending $p\in C$ to the 2 dimensional quotient space $E_p$ of
$H^0(E)$.
\begin{lemma}
$H^1(C, E)=0$ and $E$ is globally generated.
\end{lemma}
\begin{proof}
$H^1(E)\simeq H^0(E^*\otimes K_C) \not= 0$ implies the existence of
a nonzero homomorphism $E \rightarrow \Oo_C(K_C)$ which contradicts
the semistability of $E$. Now, by the same argument, we have
$H^1(E(-p))=0$ for all $p\in C$. From the long exact sequence of the
following sequence
$$0\rightarrow E(-p) \rightarrow E \rightarrow E_p \rightarrow 0, $$
we obtain the surjective evaluation map $H^0(E) \rightarrow E_p$,
which implies the global generation of $E$.
\end{proof}
In fact, the morphism $\phi$ fits in the following diagram
$$\xymatrix{ C \ar@{^{(}->}[d]_{|3K_C|} \ar[r]^{\phi} &Gr (H^0(E), 2) \ar@{^{(}->}[d]^{\theta}\\
             \PP H^0(3K_C)^* \ar@{-->}[r]^{\PP \lambda^*} & \PP (\bigwedge^2 H^0(E)^*)}$$
where $\theta$ is the Plucker embedding and $\PP \lambda^*$ comes
from the dual of the following homomorphism
$$\lambda : \bigwedge^2 H^0(E) \rightarrow H^0( \bigwedge^2 E)\simeq H^0(3K_C).$$
By the following lemma, $\PP \lambda^*$ is an embedding and so is
$\phi$ for general $E$.

\begin{lemma}\label{generate}
The homomorphism $\lambda$ is surjective for general $E\in SU_C(2,
3K_C)$.
\end{lemma}
\begin{proof}
If $E$ is stable, then by Nagata-Severi theorem \cite{Lange}, we have the
following exact sequence for $E(-K_C)$,
$$0\rightarrow \Oo (D) \rightarrow E(-K_C) \rightarrow \Oo (K_C-D)
\rightarrow 0, $$ where $D$ is a divisor of degree 1. For general
$E$, we have $H^0(E(-K_C))=0$, i.e. we can assume that
$H^0(\Oo(D))=0$, i.e. $D$ is non-effective.\\
Let $L=\Oo(K_C+D)$ and $F=\Oo (2K_C-D)$. Then we have
$$0\rightarrow L \rightarrow E \rightarrow F \rightarrow 0.$$
Note that $h^0(L)=3$, $h^0(F)=5$ and $h^1(L)=h^1(F)=0$ and from the
long exact sequence of the above sequence, we have
$$H^0(E) \simeq H^0(L)\oplus H^0(F)$$
and hence it is enough to show the surjectivity of the following map
$$H^0(L)\otimes H^0(F) \rightarrow H^0(L\otimes F)\simeq H^0(3K_C).$$
For every $p\in C$, $h^0(L(-p))=2+h^1(L(-p))=2+h^0(p-D)=2$ since $D$
is not effective. Hence, we can have a map from $C$ to
$Gr(2,H^0(L))$ sending $p$ to $H^0(L(-p))$. Since $Gr(2,
H^0(L))\simeq \PP_2$, we can choose $W\in Gr(2, H^0(L))$ which is
not the same as $H^0(L(-p))$ for any $p\in C$. Then by the choice of
$W$, it does not have base locus on $C$. Now consider the map
$$W \otimes H^0(F) \rightarrow H^0(3K_C)$$
By the Base-Point-Free Pencil Trick \cite{ACGH}, the kernel of this map is
isomorphic to $H^0(C, F\otimes L^{-1})$ and this is isomorphic to
$H^0(K_C-2D)$. Note that $h^0(K_C-2D)=h^0(2D)$ by the Riemann-Roch
theorem. If $h^0(2D)=0$, then $W\otimes H^0(F)$ is isomorphic to
$H^0(3K_C)$ by the counting of the dimensions. Hence, it is enough
to show that $H^0(2D)=0$ for general $E$. \\
Assume that $h^0(2D)>0$ and then $\Oo (2D)$ is an element of the
theta divisor in $\Pic ^2 (C)$. Since the map
$$\Pic ^1 (C) \rightarrow \Pic ^2 (C)$$
defined by $D\mapsto 2D$, is a finite surjective map of degree 64.
Hence the subvariety of $\Pic ^1(C)$ whose elements are $D$ such
that $h^0(D)=0$ and $h^0(2D)>0$ is of 2 dimension. For these
divisors $D$, the extensions of $\Oo (K_C-D)$ by $\Oo(D)$ is
parametrized by $\PP_3$, which means that the vector bundles which
does not satisfy $h^0(2D)=0$, are of at most 5 dimension. Hence
$h^0(2D)=0$ in general.
\end{proof}

Now, for the 5 dimensional subspace $V\subset H^0(E)$, we have the
following diagram
\begin{equation}
\xymatrix{ C \ar@{^{(}->}[d]_{|3K_C|} \ar@{-->}[r]^{\phi_V} &Gr (V, 2) \ar@{^{(}->}[d]^{\theta}\\
             \PP H^0(3K_C)^* \ar@{-->}[r]^{\PP \lambda^*} & \PP
(\bigwedge^2 V^*)}
\end{equation}
Consider a natural map
\begin{equation}
\xymatrix{\PP (\wedge^2 \Ee) \ar[d] \ar[r]^f & \PP(\wedge^2 V_7) \\
            Gr (5, V_7), }
\end{equation}
where $\Ee$ is the universal subbundle, $V_7$ is a 7-dimensional
vector space and $Gr(5, V_7)$ is the Grassmannian of 5-dimensional
subspaces of $V_7$. Over $[V_5]\in Gr(5, V_7)$, the fibre $\wedge^2
V_5$ is linearly embedded into $\wedge^2 V_7$.

\begin{lemma}
The image of $f$ is the secant variety of $Gr(2, V_7)\subset
\PP(\wedge^2 V_7)$ and its dimension is equal to 17.
\end{lemma}

\begin{proof}
Let $[x]\in \Image(f)$, i.e. there exists a $V_5$ such that $x\in
\wedge^2 V_5$. Consider $G=Gr (2,V_5)\subset \PP(\wedge^2 V_5)$ and
since the secant variety of $G$ is $\PP(\wedge^2 V_5)$, we can
express $x$ by
$$(v\wedge w) ~~\text{or}~~ ( v_1\wedge v_2 + v_3\wedge v_4),$$
which proves that $\Image(f)$ is contained in the secant variety of
$Gr(2, V_7)$.

Now we show the inclusion $\Sec (Gr (2, V_7))\hookrightarrow \Image
(f)$. Assume that $x$ is a general point in the secant variety. This
means that
$$x=v_1\wedge v_2 + v_3 \wedge v_4,$$
where $U=<v_1, v_2, v_3, v_4>$ is a 4-dimensional space. For any
$V_5\supset U$, we have $x\in \wedge^2 V_5$. This shows that
$$\Sec (Gr (2, V_7))=\Image (f),$$
since both sides are closed subvarieties of $\PP(\wedge^2 V_7)$. Also the set of such $V_5$ is 2-dimensional and $\dim
f^{-1}([x])=2$. Hence the dimension of $\Image (f)$ is 17, since
$\dim (\PP (\wedge^2 \Ee ))=19$.
\end{proof}
\begin{remark}
$Gr (2, V_7)$ is a Scorza variety of defect $\delta=4$ \cite{Z}. So,
it is known that $\dim \Sec (Gr (2, V_7))=17$.
\end{remark}

\begin{lemma}\label{lem2}
For general $E\in SU_C(2, 3K_C)$ and general 5-dimensional vector
subspace $V\subset H^0(E)$, the restriction of $\lambda$ to
$\wedge^2 V$,
$$\lambda : \bigwedge^2 V \rightarrow H^0(3K_C)$$
is an isomorphism.
\end{lemma}
\begin{proof}
In the proof of (\ref{generate}), let
$$V_7:=W\oplus V_5,$$
where $V_5 \simeq H^0(F)$. In fact, we can take any $V_5\subset V_7$
with $V_5 \cap H^0(L)=0$. Then, the restriction of $\lambda$ to
$\wedge^2 V_7$ is also surjective. Let $K=\ker (\lambda)$ be the
11-dimensional subspace of $\wedge^2 V_7$. Consider an incidence
variety $\mathcal{R}\subset Gr (5,V_7)\times \PP(K)$
\begin{equation}
\Rr=\{ (V_5, [x])~~|~~x\in \wedge^2 V_5 \cap K\}.
\end{equation}
We have the following diagram
\begin{equation}
\xymatrix{ & \Rr \ar[ld]_{pr_1} \ar[rd]^{pr_2}\\
             Gr (5, V_7)  && \PP(K)}.
\end{equation}
It is enough to show that the map $pr_1$ is not dominant, which
means that for the general $V_5\subset V_7$ not in the image of
$pr_1$, we have the surjection in the assertion. Assume that $pr_1$
is dominant, then
$$\dim (\Rr)\geq 10.$$
 If we consider the following map again
$$\xymatrix{\PP(\wedge^2 \Ee) \ar[d] \ar[r]^f & \PP(\wedge^2 V_7) \\
            Gr (5, V_7), }$$ then the image of $pr_2$ in $\PP (K)$
            is the intersection of $\Image (f)=\Sec (Gr (2, V_7))$
            with $\PP (K)$ in $\PP (\wedge^2 V_7)\simeq \PP_{20}$.
            Since $\Image(f)$ is 17-dimensional, we have
$$7 \leq \dim \Image (pr_2) \leq 10.$$
It is clear that $\PP (K)$ contains a point in $\Sec (Gr (2, V_7))$,
but not in $Gr (2, V_7)$. The fibre over this point in $\Rr$ is
isomorphic to $Gr (1, 3)\simeq \PP_2$. Thus the dimension of $\Image
(pr_2)$ is greater than 7.

Now assume that $\dim \PP(K) \cap \Sec (Gr (2, 7)) \geq 8$. In the
proof of (\ref{generate}), we have
$$K \cap (W\wedge V_5) = (0),$$
if $V_5\cap W=(0)$. If $V_5\cap W \not=(0)$, the intersection is
always $[\wedge^2 W]$. Let us consider the canonical map
$$s: W \otimes V_7/W \rightarrow \wedge^2 V_7/ \wedge^2 W.$$
For all $V_5$ with $V_5\cap W=(0)$, the images in $\wedge^2
V_7/\wedge^2 W$ are same as a 10-dimensional vector space. If we
take the preimage of this space in $\wedge^2 V_7$, then it is the
union of $W\wedge V_5$ for all $V_5$, which is now an 11-dimensional
space. Note that $K \cap (W\wedge V_5)=[\wedge^2 W]$ if $W \cap V_5
\not =(0)$. Let us denote by $D$ the projectivization of the
preimage of $s(W \otimes V_7/W)$ in $\wedge^2 V_7$. Then $D$ is a
10-dimensional subvariety of $\PP (\wedge^2 V_7)$ and it intersects
with $\PP (K)$ at the unique point $[\wedge^2 W]$. In fact, $D$ is
the projective tangent space $\PP T_{[W]}Gr(2, V_7)$ of $Gr (2,
V_7)$ at $[W]$ in $\PP (\wedge^2 V_7)$. Recall that

\begin{equation}
\left.
    \begin{array}{ll}
      T_{[W]}Gr (2, V_7) = \Hom (W, V_7/W) \simeq W^* \otimes V_7/W \\
      T_{[\wedge^2 W]}\PP (\wedge^2 V_7)= \Hom (\wedge^2 W, \wedge^2
      V_7/\wedge^2 W)
    \end{array}
  \right.
\end{equation}
The differential map of the Pl\"{u}cker embedding at $[W]$ is
defined as follow: $x=w^* \otimes e \in T_{[W]}Gr (2, V_7)$ is sent
to the map
$$ w_1 \wedge w_2 \mapsto s( (w^*(w_1)w_2-w_1w^*(w_2))\otimes e),$$
where $W=<w_1, w_2>$. This explains the assertion.

Now since the union of the secant lines of $Gr (2, V_7)$ passing
through $[\wedge^2 W]$ is 11 dimensional and $\PP (K) \cap \Sec (Gr
(2, V_7))$ is of dimension $\geq 8$, we can pick an element $[U]\in
\PP(K) \cap Gr (2 , V_7)$ and then the secant line
$\overline{[U][W]}$ lies in $\PP(K)$. From the condition on $W$, $U$
and $W$ span a 4-dimensional subspace of $V_7$. In particular,
general points on the secant line $\overline{[U][W]}$ are
indecomposable. Let $p$ be such a point. Since $\mathrm{Sing}( \Sec
(Gr (2, V_7)))=Gr (2, V_7)$ \cite{Z}, the dimension of $T_p (\Sec
(Gr (2, V_7)))$ is 17. Note that
\begin{equation}
T_p (\Sec (Gr (2, V_7))) = <T_{[W]}G, T_{[U]}G>.
\end{equation}
Since $$T_p (\PP(K) \cap \Sec (Gr (2, V_7)))=\PP(K) \cap T_p (\Sec
(Gr(2, V_7)))$$ is at least 8-dimensional, $\PP (K)$ intersects
$T_{[W]}G$ along at least 1-dimensional subspace, which is
contradiction because $\PP(K) \cap D$ is a single point.
\end{proof}
From the previous lemma, we have the following commutative diagram
\begin{equation}\label{main}
\xymatrix{\PP_2\simeq \PP H^0(K_C)^* \ar^{v_3}@{^{(}->}[rr] &&\PP
H^0(3K_C)^*\ar@{^{(}->}[d]\\ C\ar@{^{(}->}[u] \ar[r] & Gr (H^0(E),
2)\ar@{-->}[d] \ar@{^{(}->}[r] & \PP (\wedge^2
H^0(E)^*)\ar@{-->}[d] \\
& Gr (V,2) \ar@{^{(}->}[r] &\PP (\wedge^2 V^*),}
\end{equation}
where the composite of
the two vertical maps on the right,
\begin{equation}
\PP H^0(3K_C)^* \hookrightarrow \PP (\wedge^2 H^0(E)^*)
\dashrightarrow \PP (\wedge^2 V^*),
\end{equation}
is an isomorphism and $v_3$ is the 3-tuple Veronese embedding, i.e. $v_3$ is given by the complete linear system $|\Oo_{\PP_2}(3)|$. In particular, $C$ is embedded into $Gr (V,2)$. Note that $C$ is non-degenerate in $\PP_9\simeq \PP (\wedge^2 V^*)$ due to the Riemann-Roch theorem and the Noether theorem.

\begin{corollary}
General element $E$ in $SU_C(2, 3K_C)$ is generated by 5-dimensional
subspace of $H^0(E)$.
\end{corollary}

\section{Embedding the projective plane into Grassmannian}
In the diagram \ref{main}, the projective plane $\PP
H^0(K_C)^*\simeq \PP _2$ is embedded into the projective space $\PP
(\wedge^2 V^*)\simeq \PP _9$ by the Pl\"{u}cker embedding.
\begin{lemma}
For general $E\in SU_C(2, 3K_C)$, there exists a 5-dimensional
vector subspace $V\subset H^0(E)$ such that $\PP H^0(K_C)^*$ is
embedded into $Gr (V,2)$ in the diagram \ref{main}.
\end{lemma}
\begin{proof}
Let $V\subset H^0(E)$ be a 5-dimensional subspace selected in \ref{lem2} and assume that $\PP H^0(K_C)^*$ is not embedded into $Gr(V,2)$. Recall that $Gr(V,2)$ is cut out by 4-dimensional projectively linear family of quadrics of rank 6 in $\PP_9$ whose singular locus is $\PP_3$ contained in $Gr(V,2)$ as the Schubert variety of lines through a point corresponding to the quadric in $\PP_4$ \cite{SR}. Let $Q(p)$ be one of the quadrics of rank 6 containing $Gr(V,2)$ which does not contain $S$ where $p$ is a point in $\PP_4$ and $S$ is the image of $\PP_2$ by $v_3$. Since $v_3^{-1}(Q(p))$ is a plane sextic curve, we have
$$v_3^{-1}(Q(p))=C+C',$$
where $C'$ is a conic. First, assume that $Gr(V,2)\cap S=C+C'$. If we consider the incidence variety $Z_{C}=\{ (l,x)| x\in l\}\subset C \times \PP_4$, we have a diagram
$$\xymatrix{ & Z_C \ar[ld]_{p} \ar[rd]^{q} \\
C && \PP_4}$$
and let $S_C$ be the image of $q$ in $\PP_4$. If $S_C$ is degenerate, i.e. there exists a hyperplane $\PP_3\subset \PP_4$ containing $S_C$, then $C$ is contained in some grassmannian $Gr(4,2)\subset Gr(V,2)$ and in particular, $C$ is contained in $\PP_5$, the Pl\"{u}cker space of $Gr(4,2)$, which is contradiction to the non-degeneracy of $C$ in $\PP_9$. Similarly we can define $Z_{C'}$ and $S_{C'}$. Recall the well known fact that
$$\deg (C)=\deg (S_C)\cdot \deg(q).$$
If $\deg(S_C)=1$, i.e. $S_C$ is a plane in $\PP_4$, then $C$ must be contained in $\PP_3(p)$, the singular locus of a quadric $Q(p)$ for $p\in S_C$, which is contradiction to the fact that $C\subset \PP_9$ is nondegerate. Hence $\deg(S_C)\geq 2$ and so $\deg(q)\leq 6$. This implies that the number of points in $\PP_3(p)\cap C$ is less than 7 for $p\in S_C$. Since the intersection of $S_C$ and $S_{C'}$ is at most 1-dimensional in $S_C$, we have still 2-dimensional choices for $p$ for which $\PP_3(p) \cap (C+C')=\PP_3(p)\cap C$ is less than 7 points. We can also have the same conclusion on the intersection number of $\PP_3(p)\cap (C+C')$ in the case when $Gr(V,2)\cap S$ is the proper subset of $C+C'$ since it still contains $C$.
Now choose $p\in \PP_4$ such that the singular locus $\PP_3(p)$ of $Q(p)$ meets $C+C'$ with $k$ points where $0<k<7$. We have the following commutative diagram
$$\xymatrix{ \PP_3(p) \ar@{^{(}->}[d]\\
Gr(V,2) \ar@{-->}[d] \ar@{^{(}->}[r] & \PP (\wedge^2 V^* )\ar@{-->}[d]  &S\ar@{-->}[d] \ar@{_{(}->}[l] & C+C' \ar@{_{(}->}[l] \ar[d]\\
Gr(4,2) \ar@{^{(}->}[r] &\PP_5 & \overline{S}\ar@{_{(}->}[l] &\overline{C+C'},\ar@{_{(}->}[l]}$$
where $\overline{S}$, $\overline{C+C'}$ are the image of $S$, $C+C'$, respectively, via the projection and the image of $Gr(V,2)$ lies in the image of the quadric $Q$, i.e. the Grassmannian $Gr(4,2)\subset \PP_5$. Let $Q'$ be another quadric cutting $Gr(V,2)$ with singular locus $\PP_3'$. Since $\PP_3 \cap \PP_3'$ is a single point, The image of $Q'$ by the projection is $\PP_5$. Thus the image of $Gr(V,2)$ is $Gr(4,2)$. Note that the degree of $\overline{C+C'}$ is $18-k$ and the degree of $\overline{S}$ is $9-k$ since $\PP_3(p)\cap S=\PP_3(p)\cap (C+C')$. If $Q(p)$ contains $S$ for all such $p\in S_C$, then all quadrics containing $Gr(V,2)$ of rank 6, should contain $S$ since $S_C$ is nondegerate in $\PP_4$. In particular, $Gr(V,2)$ should contain $S$, which is against the assumption. So there exists a $p\in S_C$ for which $S$ is not contained in $Q(p)$. Thus the image of $S$ by the projection is not also contained in the image of $Q(p)$, i.e. $Gr(4,2)$.  But the degree of intersection $Gr(4,2)\cap \overline{S}$ is $2\times (9-k)<18-k$, which is contradiction to the fact that this intersection contains $\overline{C+C'}$.

\end{proof}

Let $U_V$ and $\overline{U_V}$ be the universal subbundle and
quotient bundle of $Gr (V,2)$, respectively. With the condition on
$V$ in the previous lemma, let
\begin{equation}
E_V:= v_3^* \overline{U_V},
\end{equation}
which implies that the restriction of $E_V$ to $C$ is $E$, i.e.
$E_V|_C=E$.

\begin{lemma}\label{stable}
$E_V$ is stable with the Chern classes $(3,6)$, i.e. $E_V\in
\overline{M}(3,6)$.
\end{lemma}
\begin{proof}
Since the first Chern class of $\overline{U_V}$ is the hyperplane
section of $Gr (V,2)$ in $\PP (\wedge^2 V^*)$ and $v_3$ is the
3-tuple Veronese embedding, we get $c_1(E_V)=3$.

 By the choice of $V$, we have an exact sequence,
\begin{equation}\label{ses}
0\rightarrow G \rightarrow V \otimes \Oo_{\PP_2} \rightarrow E_V
\rightarrow 0,
\end{equation}
where $G$ is the kernel of the
surjection $V \otimes \Oo_{\PP_2} \twoheadrightarrow E_V$ and $V$ is
a 5-dimensional vector subspace of $H^0(E_V)$. In particular,
$h^0(E_V)\geq 5$. By the choice of $E$, we have $h^0(
E_V(-1)|_C)=0$. From the long exact sequence of cohomology of the
following exact sequence,
$$0\rightarrow E_V(-5) \rightarrow E_V(-1) \rightarrow E_V(-1)|_C
\rightarrow 0, $$ we have
$$H^0(E_V(-5)) \simeq H^0(E_V(-1)).$$

For a line $H \subset \PP_2$, $E_V|_H \simeq \Oo_H(a)\oplus \Oo_H(3-a)$ for $a=2$ or $3$ since $E_V$ is globally generated. In particular, $h^0(E_V(-k)|_H)=0$ for $k\geq 4$. From the long exact sequence of cohomology of the following exact sequence $$ 0\rightarrow E_V(-k-1) \rightarrow E_V(-k) \rightarrow E_V(-k)|_H \rightarrow 0,$$
we have $h^0(E_V(-k-1))=h^0(E_V(-k))$ for all $k\geq 4$. Since $h^0(E_V(-k))=0$ for sufficiently large $k$, we have $h^0(E_V(-k))=0$ for $k\geq 4$ and in particular, $h^0(E_V(-1))=h^0(E_V(-5))=0$, i.e. $h^0(E_V(-k))=0$ for all $k\geq 1$. Hence the vector bundle $E_V$ is stable.

Again, let $H$ be a line in $\PP_2$. From the following exact sequence,
$$0\rightarrow E_V(-1) \rightarrow E_V \rightarrow E_V|_H
\rightarrow 0,$$ we get $h^0(E_V)\leq h^0(E_V|_H)$. Since $E_V|_H \simeq \Oo_H(a) \oplus \Oo_H(3-a)$ for $a=2$ or 3, $h^0(E_V|_H)=5$ and so $h^0(E_V)\leq 5$. Thus we obtain $h^0(E_V)=\dim V=5$.

Now from the long exact sequence of cohomology of \ref{ses}, we have
$h^0 ( \PP_2, G)=0$. If we twist  \ref{ses} by -1, we have $h^1
(\PP_2, G(-1))=0$. For any line $l\subset \PP_2$, consider the
following exact sequence
$$0\rightarrow G(-1) \rightarrow G \rightarrow G|_l \rightarrow 0.$$
From the above statement, we get $H^0(G|_l)=0$. Since
$c_1(G)=-c_1(E_V)=-3$, we have $G|_l\simeq \Oo_l(a)\oplus \Oo_l(b)
\oplus \Oo_l(c)$ with $a+b+c=-3$. The only choice from the vanishing
of $H^0(G|_l)$ is $(a,b,c)=(-1,-1,-1)$. Hence $G$ is a uniform
vector bundle of rank 3 on $\PP_2$ with the splitting type
$(-1,-1,-1)$. From the classification of such bundles \cite{Elen}, we
have
$$G \simeq \Oo_{\PP_2}(-1)^{\oplus 3}.$$
In particular, $c_2(G)=3$ and so $c_2(E_V)=6$.

\end{proof}

Since we can pick an element $E_V\in \overline{M}(3,6)$ mapping to a
general element $E\in SU_C(2, 3K_C)$, the rational map
\begin{equation}\label{dom2}
\overline{M}(3,6) \dashrightarrow SU_C(2, 3K_C)
\end{equation}
is dominant. By twisting the map (\ref{dom2}) with $\Oo_{\PP_2}
(-1)$ and $\Oo_C(-K_C)$, we have the following main theorem.
\begin{theorem}
The restriction map $$\Phi_4 : \overline{M}(1,4) \dashrightarrow
SU_C(2,K_C)$$ is dominant.
\end{theorem}

\begin{remark}
Dolgachev and Kapranov \cite{DK1} showed that the logarithmic
bundles $E(\mathcal{H})$ attached to the general hyperplane
arrangement $\mathcal{H}=(H_1, \cdots, H_6)$ in $\PP_2$, form a open
zariski subset $U \subset \overline{M}(3,6)$. For these bundles
$E(\mathcal{H})$, we have a Steiner resolution
$$0\rightarrow \Oo_{\PP_2}(-1)^{\oplus 3} \rightarrow
\Oo_{\PP_2}^{\oplus 5} \rightarrow E(\mathcal{H}) \rightarrow 0.$$
From this, we have a 5 dimensional space $V=H^0(\PP_2,
E(\mathcal{H}))$ and by the tensoring the following exact sequence
by $E(\mathcal{H})$
$$0\rightarrow \Oo_{\PP_2} (-4) \rightarrow \Oo_{\PP_2} \rightarrow \Oo_C
\rightarrow 0,$$ we can consider $V$ as a subspace of $H^0(C,
E(\mathcal{H})|_C)$, which is 8-dimensional. As we have seen already
in the proof of \ref{stable}, the bundle $E_V$ has a Steiner
resolution, pulled back from the universal exact sequence on the
Grassmannian $Gr(V,2)$. This motivates the whole argument in this
paper.
\end{remark}

Since $\overline{M}(1,4)$ is rational and the map $\Phi_4$ is
dominant, $SU_C(2, K_C)$ is unirational. It implies that $SU_C(2,
K_C)$ is rationally connected and so rationally chain-connected.
 Let $\Hh=(H_0, \cdots, H_6)$ be a general arrangement of 6 lines on
$\PP_2$ and then we can associate a logarithmic bundle $E(\Hh) \in
\overline{M}(3,6)$ to $\Hh$. It is known \cite{DK1} that the
logarithmic bundles $E(\Hh)$ form an open Zariski subset of
$\overline{M}(3,6)$ and, after twisting by $\Oo_{\PP_2}(-1)$,
$\overline{M}(1,4)$. Let $\Ff$ be a family of arrangements of 6
lines on $\PP_2$ and $E(\Ff)$ be the closure of the subvariety of
$\overline{M}(1,4)$ whose closed points correspond to $E(\Hh)
\otimes \Oo_{\PP_2}(-1)$ with $\Hh \in \Ff$.

\begin{proposition}
$SU_C(2,K_C)$ is rationally chain-connected. In fact, any two general points in $SU_C(2, K_C)$ can be connected by at most 6 rational curves which can be described explicitly. 
\end{proposition}
\begin{proof}
Let us consider a special type of arrangements of 6 lines. Let $H_0,
H_1, \cdots, H_5$ be 6 lines in general position on $\PP_2$ and $p$
be a fixed point on $H_0$ in general position. If we fix $H_1,
\cdots, H_5$, then we have a 1-dimensional family $\Ff$ of 6 lines
with $H_0$ moving. Consider a map
$$\Psi : \PP_1(\Ff) \rightarrow SU_C(2, K_C),$$
sending $\Hh$ to $E(\Hh)(-1)|_C$. Since $SU_C(2, K_C)$ is
projective, this map is a morphism \cite{Hartshorne}. Clearly $\Psi$ is not a constant
map, otherwise $\Phi_4$ is also a constant map, which is not true.
From the fact that logarithmic bundles associated to 6 lines in
general position, form an open Zariski subset of $\overline{M}(3,6)$
and $\Phi_4$ is dominant, we can find a 1-dimensional family of 6
lines $\Ff$ which maps to a rational curve on $SU_C(2, K_C)$ via
$\Psi$ for a general element of $SU_C(2, K_C)$. Furthermore, for two
general elements $E_1, E_2\in \overline{M}(3,6)$, we can find 6
families of 6 lines $\Ff_i$, $1\leq i \leq 6$, as above, such that
the arrangements corresponding to $E_1, E_2$ lie in $\Ff_1$,
$\Ff_6$, respectively and $\Ff_i \cap \Ff_{i+1} \not = \emptyset$.
From this fact with the dominance of $\Phi_4$, we can find 6
rational curves passing through two general points on $SU_C(2,
K_C)$.
\end{proof}

\begin{remark}
 Note that we can choose these rational curves not contained
in the singular locus of $SU_C(2, K_C)$ which is the Kummer variety
of $\Pic^2 (C)$. Let $\widetilde{S}$ be a desingularization by the
blow-up \cite{Kiem} and consider the proper transform of the
previous 6 rational curves on $SU_C(2, K_C)$. It shows the
rationally chain-connected of $\widetilde{S}$ and since
$\widetilde{S}$ is smooth, it implies the rational connectedness,
i.e. the chain of these 6 curves can be deformed to a rational curve
and its image on $SU_C(2, K_C)$ will give us a rational curve
through two general points.
\end{remark}

\bibliographystyle{amsplain}
\bibliography{ResSur}
\end{document}